\documentclass[reqno]{amsart}

\usepackage{url,fancyvrb,amscd,verbatim}

\newcommand{\NN}{\mathbf{N}} 
\newcommand{\ZZ}{\mathbf{Z}} 
\newcommand{\QQ}{\mathbf{Q}} 
\newcommand{\RR}{\mathbf{R}} 
\newcommand{\FF}{\mathbf{F}} 
\newcommand{\PP}{\mathbf{P}} 

\providecommand{\card}[1]{\lvert#1\rvert} 
\providecommand{\HdR}{H_{\text{dR}}}   	      	
\providecommand{\Hrig}{H_{\text{rig}}} 	      	
\providecommand{\BigOh}{O}          		
\providecommand{\SoftOh}{\tilde{O}} 		

\DeclareMathOperator{\Frob}{F}          
\DeclareMathOperator{\Gal}{Gal} 	
\DeclareMathOperator{\ord}{ord} 	
\DeclareMathOperator{\fCoKer}{coker}    
\DeclareMathOperator{\fIm}{im} 		
\DeclareMathOperator{\Spec}{Spec} 	

\newtheorem{thm}{Theorem}[section]

\newtheorem{prop}[thm]{Proposition}
\newtheorem{cor}[thm]{Corollary}
\newtheorem{defn}[thm]{Definition}

\newtheorem{rem}[thm]{Remark}

\newtheorem{assump}{Assumption}

\title{Counting points on curves using a map to $\PP^1$, II.}
\author{Jan Tuitman}
\address{KU Leuven,
         Departement Wiskunde,
         Celestijnenlaan 200B,
         3001 Leuven,
         Belgium}
\email{jan.tuitman@kuleuven.be}

\begin{document}

\begin{abstract}
We introduce a new algorithm to compute the zeta function of a curve over a finite field. 
This method extends previous work of ours to all curves for which a good lift to characteristic zero 
is known. We develop all the necessary bounds, analyse the 
complexity of the algorithm and provide a complete implementation.
\end{abstract}

\maketitle

\section{Introduction}

Let $\FF_q$ denote the finite field of characteristic $p$ and cardinality $q=p^n$. Suppose that $X$ is a smooth projective algebraic curve of genus $g$ over $\FF_q$. Recall that the zeta function 
of $X$ is defined as
\[
Z(X,T)=\exp\left(\sum_{i=1}^{\infty} \card{X(\FF_{q^i})} \frac{T^i}{i} \right).
\]
It follows from the Weil conjectures that $Z(X,T)$ is of the form
\[
\frac{\chi(T)}{(1-T)(1-qT)},
\]
with $\chi(T) \in \ZZ[T]$ a polynomial of degree $2g$, the inverse roots of which have complex absolute value $q^{\frac{1}{2}}$ 
and are permuted by the map $t \rightarrow q/t$. 

Kedlaya~\cite{kedlaya} showed that $Z(X,T)$ can be determined efficiently, in the case when 
$X$ is a hyperelliptic curve and the characteristic $p$ is odd, by explicitly computing the action 
of Frobenius on the $p$-adic cohomology of~$X$. This was then extended by others to characteristic~$2$~\cite{dvhyp}, 
superelliptic curves~\cite{gaugu}, $C_{ab}$ curves~\cite{dv} and nondegenerate curves~\cite{cdv}. 
In~\cite{tuitman} we proposed a much more general and practical extension of Kedlaya's algorithm. The goal of
this paper is to further improve this algorithm.

The algorithm from \cite{tuitman} can be applied to generic, or in other words random, equations $Q$. However, there are equations to
which it cannot be applied including some very interesting examples. For example, when $Q$ is (the reduction 
at some prime number $p$ of) one of the defining equations computed for modular curves in \cite{sutherland,yang}, it turns out that the 
algorithm can almost never be applied. The reason is that in \cite{tuitman} we 
assume that $Q$, or rather its lift $\mathcal{Q}$ to characteristic zero, defines a smooth curve in the affine 
$(x,y)$-plane, i.e. that all the singularities of the plane curve defined by $\mathcal{Q}$ lie at infinity. 
In this paper we improve the algorithm from~\cite{tuitman} in (at least) two ways.

First, we eliminate the assumption that $\mathcal{Q}$ does not have any singularities in the affine $(x,y)$-plane.
As a consequence, our algorithm can now be applied to any curve for which we know a good lift to characteristic
zero in the sense of Assumption~\ref{assump:goodlift} below. In particular, for any smooth curve 
defined over the rational numbers, the algorithm can now be applied to the reduction of the curve modulo~$p$ for 
almost all prime numbers $p$. Compared with~\cite{tuitman} we have also reformulated Assumption~\ref{assump:goodlift}
and added some discussion on when it is satisfied.

Second, we give much better bounds for the $p$-adic precision required for obtaining provably correct results. In~\cite{tuitman}
we were mainly interested in obtaining the correct complexity estimate and not sharp precision bounds. In Section~\ref{sec:precision} 
we use the Newton-Girard identies and (log)-crystalline cohomology to obtain better precision bounds that are usually sharp. 

The time complexity of the algorithm is $\SoftOh(p d_x^6 d_y^4 n^3)$ by Theorem~\ref{thm:time}, 
and the space complexity $\SoftOh(p d_x^4 d_y^3 n^3)$ by Theorem~\ref{thm:space} (under one additional rather harmless
assumption which is Assumption~\ref{assump:ordW} below) as was the case in~\cite{tuitman}. Note that the time and space complexities 
of our algorithm are quasilinear in $p$ and hence not polynomial in the
size of the input which is $\log(p) d_x d_y n$. This is also the case for Kedlaya's 
algorithm and the algorithm from \cite{cdv} for example. However, for hyperelliptic curves, the dependence on $p$ of the 
time and space complexities of Kedlaya's algorithm has been improved to $\SoftOh(p^{1/2})$~\cite{harvey1} and average 
polynomial time~\cite{harvey2} by Harvey. It is an interesting open problem whether these ideas can be used to
improve the dependence on $p$ of the complexity of our algorithm as well.

Most of the theorems and propositions in this paper are very similar to corresponding ones in~\cite{tuitman}. However, there are lots of
small changes in many different places. To limit the amount of text overlap, we refer to~\cite{tuitman} 
whenever a proof is the same or very similar. We have updated our implementation in Magma \cite{magma}. The code can 
be found in the packages \verb{pcc_p{ and \verb{pcc_q{ at our webpage\footnotemark \footnotetext{\url{https://perswww.kuleuven.be/jan_tuitman}}.

The author was supported by FWO-Vlaanderen. We thank Peter Bruin, Wouter Castryck, Florian Hess and Kiran Kedlaya for helpful discussions.

\section{Lifting the curve and Frobenius}

\label{sec:lift} 

Recall that $X$ is a smooth projective algebraic curve of genus $g$ over the finite field $\FF_q$ of characteristic $p$ and cardinality $q=p^n$. Let $\QQ_p$ denote
the field of $p$-adic numbers and $\QQ_q$ its unique unramified extension of degree $n$. As usual, let 
$\sigma \in \Gal(\QQ_q/\QQ_p)$ denote the unique element that lifts the $p$-th power Frobenius map
on $\FF_q$ and let $\ZZ_q$ denote the ring of integers of $\QQ_q$, so that $\ZZ_q/p\ZZ_q \cong \FF_q$. 
Let $x:X \rightarrow \mathbf{P}^1_{\FF_q}$ be a finite separable map 
of degree $d_x$ and $y:X \rightarrow \mathbf{P}^1_{\FF_q}$ a rational function that generates 
the function field of $X$ over $\FF_q(x)$, such that $Q(x,y)=0$ where $Q \in \FF_q[x,y]$ 
is irreducible and monic in the variable $y$ of degree $d_x$. The degree of $Q$ 
in the variable $x$ (which is also the degree of the map $y$) will be denoted by $d_y$. Let $\mathcal{Q} \in \ZZ_q[x,y]$ be 
some lift of $Q$ that contains the same monomials in its support as $Q$ and is still monic in $y$. 

\begin{defn} \label{defn:Delta}
We let $\Delta(x) \in \ZZ_q[x]$ denote the discriminant of 
$\mathcal{Q}$ with respect to the variable $y$, define $r(x) \in \ZZ_q[x]$ 
as the squarefree polynomial $r=\Delta/(\gcd(\Delta,\frac{d\Delta}{dx}))$
and let $m \in \NN$ be the least positive integer such that there exist
a polynomial $g(x) \in \ZZ_q[x]$ that satisfies $r(x)^m = g(x) \Delta(x)$.
\end{defn}

We will denote $\mathcal{S}= \ZZ_q[x,1/r]$ and $\mathcal{R}= \ZZ_q[x,1/r,y]/(\mathcal{Q})$.
Moreover, we write $\mathcal{V}= \Spec \mathcal{S}$, $\mathcal{U}= \Spec \mathcal{R}$,
so that $x$ defines a finite \'etale morphism from $\mathcal{U}$ to $\mathcal{V}$. 
We let $U=\mathcal{U} \otimes_{\ZZ_q} \FF_q$, $V=\mathcal{V} \otimes_{\ZZ_q} \FF_q$ 
denote the special fibres and $\mathbb{U}=\mathcal{U} \otimes_{\ZZ_q} \QQ_q$, 
$\mathbb{V}=\mathcal{V} \otimes_{\ZZ_q} \QQ_q$ the generic fibres of
$\mathcal{U}$ and $\mathcal{V}$. Finally, $\FF_q(x,y)$ will denote
the field of fractions of $\mathcal{R} \otimes \FF_q$ and $\QQ_q(x,y)$ the field of fractions
of $\mathcal{R} \otimes \QQ_q$.

\begin{assump} \label{assump:goodlift} We will assume that:
\begin{enumerate}
\item Matrices $W^0 \in Gl_{d_x}(\ZZ_q[x,1/r])$ and
$W^{\infty} \in Gl_{d_x}(\ZZ_q[x,1/x,1/r])$ are given such that, if we denote 
$b^0_j = \sum_{i=0}^{d_x-1} W^0_{i+1,j+1} y^i$ and 
$b^{\infty}_j = \sum_{i=0}^{d_x-1} W^{\infty}_{i+1,j+1} y^i$ for all $0 \leq j \leq d_x-1$, then:
\begin{enumerate}
\item $[b^{0 \;}_0,\ldots,b^{0 \;}_{d_x-1}]$ is an integral basis for $\QQ_q(x,y)$ over $\QQ_q[x]$ and its reduction modulo~$p$ is an integral basis for $\FF_q(x,y)$ over $\FF_q[x]$,
\item $[b^{\infty}_0,\ldots,b^{\infty}_{d_x-1}]$ is an integral basis for $\QQ_q(x,y)$ over $\QQ_q[1/x]$ and its reduction modulo~$p$ is an integral basis for $\FF_q(x,y)$ over $\FF_q[1/x]$.
\end{enumerate}
\noindent Let $W \in Gl_{d_x}(\ZZ_q[x,1/x])$ be the change of basis matrix defined by $W = (W^0)^{-1} W^{\infty}$ and denote
\begin{align*}
\mathcal{R}^0        &= \ZZ_q[x]b^{0}_0 \; \; \; \; \; +\ldots+\ZZ_q[x]b^{0}_{d_x-1}, \\
\mathcal{R}^{\infty} &= \ZZ_q[1/x]b^{\infty}_0+\ldots+\ZZ_q[1/x]b^{\infty}_{d_x-1}.
\end{align*}
Note that these are rings (even $\ZZ_q[x]$ and $\ZZ_q[1/x]$-algebras, respectively).
\item The discriminant of $r(x)$ is a unit.
\item The discriminants of the finite $\ZZ_q$-algebras $\mathcal{R}^0/(r(x))$ and $\mathcal{R}^{\infty}/(1/x)$ are units.
\end{enumerate}
\end{assump}

\begin{rem}
Note that the extra assumption from~\cite{tuitman} (that we are eliminating here) was that $W^0$
is the identity matrix.
\end{rem}

Geometrically, Assumption~\ref{assump:goodlift} says that the finite \'etale morphism $x: \mathcal{U} \rightarrow \mathcal{V}$
admits a good compactification. More precisely:

\begin{prop} \label{prop:goodlift} \mbox{ }
\begin{enumerate} 
\item There exists a smooth relative divisor $\mathcal{D}_{\mathbf{P}^1}$ on $\mathbf{P}^1_{\ZZ_q}$ 
such that $\mathcal{V} = \mathbf{P}^1_{\ZZ_q} \setminus \mathcal{D}_{\mathbf{P}^1}$.  
\item There exists a smooth proper curve $\mathcal{X}$ over $\ZZ_q$ and a smooth relative 
divisor $\mathcal{D}_{\mathcal{X}}$ on $\mathcal{X}$ such that $\mathcal{U} = \mathcal{X} \setminus \mathcal{D}_{\mathcal{X}}$. 
\end{enumerate}
\end{prop}
\begin{proof}
We can glue $\Spec{\mathcal{R}^0}$ and $\Spec{\mathcal{R}^{\infty}}$ together along $\mathcal{U}$ to obtain a curve $\mathcal{X}$
over $\ZZ_q$. Note that $\mathcal{R}^0$ and $\mathcal{R}^{\infty}$ are clearly flat over $\ZZ_q$, so smoothness follows from 
regularity of the special and generic fibres, which is a consequence of the first part of Assumption~\ref{assump:goodlift}. 
The complement $\mathcal{D}_{\PP^1}$ of $\mathcal{V}$ in $\PP^1_{\ZZ_q}$ is the union of the zero locus of $r(x)$ and the point~$\infty$ and
is \'etale (hence smooth) over $\ZZ_q$ by the second part of Assumption~\ref{assump:goodlift}. Finally, the complement of 
$\mathcal{U}$ in $\mathcal{X}$ is the union of the zero locus of $r(x)$ and $x^{-1}(\infty)$ and is \'etale (hence smooth) over $\ZZ_q$ by 
the third part of Assumption~\ref{assump:goodlift}.
\end{proof}

We write $\mathbb{X}=\mathcal{X} \otimes \QQ_q$ for the generic fibre of $\mathcal{X}$. Note that $\mathcal{X} \otimes \FF_q \cong X$
by construction. Moreover, $z_P$ will denote an \'etale local coordinate and $e_P$ the ramification index of the map $x$ at a point
$P \in \mathcal{X} \setminus \mathcal{U}$.

Note that in~\cite{tuitman}, Proposition~\ref{prop:goodlift} was itself the main assumption and not a consequence of it. However, there
we still needed to assume that $W^{0}$ and $W^{\infty}$ were known (actually we restricted to the case where $W^0$ could be taken to be the 
identity matrix). Stating Assumption~\ref{assump:goodlift} as above and deriving Proposition~\ref{prop:goodlift} as a consequence is
simpler and shows more clearly how to check explicitly that a lift of $X$ given by $\mathcal{Q}$ and the matrices $W^0$, $W^{\infty}$
is suitable for the algorithm. Since Assumption~\ref{assump:goodlift} is the only remaining (but essential) assumption for our algorithm to work, 
let us analyse it in some more detail now.

It is natural to ask when a lift $\mathcal{Q}$ and matrices $W^0$, $W^{\infty}$ satisfying Assumption~\ref{assump:goodlift} exist for a given $Q$. 
From the theory of the tame fundamental group~\cite[Expos\'e XIII, Section~$2$]{sga1}, it should follow that this is the case when the map $x:X \rightarrow \PP^1_{\FF_q}$ is 
tamely ramified. Since any curve of characteristic $p > 2$ is a tame cover of the projective line~\cite[Theorem $8.1$]{fulton} (at least after 
extending the base field), by varying $Q$ our method should apply to any curve in characteristic $p > 2$. However, in our algorithm we need to 
know all of these polynomials and matrices explicitly, knowing that they exist is of little use. 

We would like to have an algorithm that given $Q$ finds a lift $\mathcal{Q}$ and matrices $W^0$, $W^{\infty}$ satisfying
Assumption~\ref{assump:goodlift} when they exist. However, even for the simpler problem of finding a smooth lift 
$\mathcal{X}$ of a curve $X$ (to some finite $p$-adic precision $N$) we have not found an effective solution in the literature except
in some special cases like complete intersections in projective space or nondegenerate curves for which it is trivial. Therefore, the problem
of finding a lift $\mathcal{Q}$ and matrices $W^0$, $W^{\infty}$ satisfying Assumption~\ref{assump:goodlift} is probably hard
in general. Note that other point counting algorithms using $p$-adic cohomology also need a good lift to characteristic~$0$,
but almost always restrict to nondegenerate curves or hypersurfaces, for which it is easy to find one. The only exception to this that 
we know of is~\cite{dvhyp}, where indeed quite a lot of effort goes into finding a good lift to characteristic~$0$ for hyperelliptic 
curves in characteristic~$2$.

Although it is probably hard to find a lift $\mathcal{Q}$ and matrices $W^0$, $W^{\infty}$ satisfying Assumption~\ref{assump:goodlift}
in general, the following strategy is often succesful. Let $K$ be a number field of degree~$n$ in which $p$ is inert and 
let $\mathcal{O}_{K}$ denote its ring of integers. Then we can identify the residue field $\mathcal{O}_{K}/p\mathcal{O}_{K}$ with 
$\FF_q$ and the $p$-adic completion of $\mathcal{O}_{K}$ with $\ZZ_q$. We first try to find a lift $\mathcal{Q} \in \mathcal{O}_K[x,y]$ 
that defines a function field of genus equal to the genus of $X$. Over a number field efficient algorithms to compute integral bases in 
function fields are available~\cite{hess,bauch}. We can simply run such an algorithm, hope that the matrices $W^0$,$W^{\infty}$ and their inverses are $p$-adically 
integral and that the second and third condition of Assumption~\ref{assump:goodlift} are also satisfied. 
Together with W. Castryck we have recently shown that (in odd characteristic) this strategy works for (almost) all curves of genus at most~$5$ 
and most trigonal and tetragonal curves, even if we impose that the degree $d_x$ of the morphism $x$ is as small as possible, i.e. equals the 
gonality of the curve~\cite{castrycktuitman}. 

Note that if we start from $\mathcal{Q} \in \ZZ[x,y]$ and compute $W^0$, $W^{\infty}$ over $\QQ$, then Assumption~\ref{assump:goodlift} will
be satisfied for all but a finite number of primes~$p$ (by generic smoothness). Therefore, for any curve over $\QQ$ our algorithm applies modulo all but a finite number
of primes~$p$ and a similar statement holds over number fields. So our algorithm can in principle be applied to computing $L$-series of general
curves, although this will not be very efficient since the time complexity per prime~$p$ is quasilinear in $p$.

To summarise the discussion above: existence of a lift $\mathcal{Q}$ and matrices $W^0$, $W^{\infty}$ is usually not a problem (in odd characteristic),
but it is not clear how to find them explicitly in general. In some (quite general) special cases we can almost always find a suitable lift, for example 
for curves of genus at most~$5$ and most nondegenerate, trigonal or tetragonal curves~\cite{castrycktuitman}. Finally, the lifting problem can also be circumvented by starting 
from a curve that is already defined over a number field, which is still very interesting from the point of view of computing zeta functions. \\

We now move on to the first part of the algorithm, which is lifting the Frobenius map.

\begin{prop} \label{prop:s}
Let $\mathcal{A}$ denote the ring $\ZZ_q[x,y]/(\mathcal{Q})$. Then the quotient
\[
s(x,y)=\Delta(x)/\frac{\partial \mathcal{Q}}{\partial y}
\]
exists in $\mathcal{A}$.
\end{prop}
\begin{proof}
For $k \in \NN$, we let $W_k$ denote the free $\ZZ_q[x]$-module of polynomials in $\ZZ_q[x,y]$ of degree
at most $k-1$ in the variable $y$. Let $\Sigma$ be the matrix of the $\ZZ_q[x]$-module homomorphism:
\begin{align} \label{eq:alphabeta}
W_{d_x-1} \oplus W_{d_x} &\rightarrow W_{2d_x-1}, &(a,b) \mapsto a \mathcal{Q} + b \frac{\partial \mathcal{Q}}{\partial y},
\end{align}
with respect to the bases $[1,y,\dotsc,y^{d_x-2}]$, $[1,y,\dotsc, y^{d_x-1}]$ and $[1,y,\dotsc,y^{2d_x-2}]$. 
By definition we have $\Delta = \det(\Sigma)$, so that $\Delta$ is contained in the image of~\eqref{eq:alphabeta} and
$\Delta(x)/\frac{\partial \mathcal{Q}}{\partial y}$ exists in $\mathcal{A}$. 
\end{proof}

\begin{defn} We denote the ring of overconvergent functions on $\mathcal{U}$ by
\[
\mathcal{R}^{\dag} = \ZZ_q \langle x, 1/r, y \rangle^{\dag}/(\mathcal{Q}). 
\]
Note that $\mathcal{R}^{\dag}$ is a free module of rank $d_x$ over $\mathcal{S}^{\dag}=\ZZ_q \langle x, 1/r \rangle^{\dag}$ 
and that a basis is given by $[y^0, \dotsc, y^{d_x-1}]$. A Frobenius lift 
$\Frob_p:\mathcal{R}^{\dag} \rightarrow \mathcal{R}^{\dag}$ is defined as a $\sigma$-semilinear 
ring homomorphism that reduces modulo $p$ to the $p$-th power Frobenius map. 
 \end{defn}
 
\begin{thm} \label{thm:froblift} There exists a Frobenius lift $\Frob_p: \mathcal{R}^{\dag} \rightarrow \mathcal{R}^{\dag}$ for which
$\Frob_p(x)=x^p$. 
\end{thm}

\begin{proof}
Let notation be as in Definition~\ref{defn:Delta} and Proposition~\ref{prop:s}.
Define sequences $(\alpha_i)_{i \geq 0}$, $(\beta_i)_{i \geq 0}$, 
with $\alpha_i \in S^{\dag}$ and $\beta_i \in \mathcal{R}^{\dag}$, 
by the following recursion:
\begin{align*}
&\alpha_0=\frac{1}{r^{p}}, \\
&\beta_0=y^p, \\
&\alpha_{i+1}=\alpha_i(2-\alpha_i r^{\sigma}(x^p)) &\pmod{ \; p^{2^{i+1}}}, \\ 
&\beta_{i+1}= \beta_i - \mathcal{Q}^{\sigma}(x^p,\beta_i) s^{\sigma}(x^p, \beta_i) g^{\sigma}(x^p) \alpha_i^m &\pmod{ \; p^{2^{i+1}}}.
\end{align*}
Then one easily checks that the $\sigma$-semilinear ringhomomorphism $\Frob_p: \mathcal{R}^{\dag} \rightarrow \mathcal{R}^{\dag}$ defined by
\begin{align*}
\Frob_p\bigl(x\bigr)&=x^p, &\Frob_p\left(1/r\right)&=\lim_{i \rightarrow \infty} \alpha_i,
&\Frob_p\bigl(y\bigr)&=\lim_{i \rightarrow \infty} \beta_i,
\end{align*}
is a Frobenius lift.
\end{proof}

\begin{rem}
Comparing to~\cite{tuitman}, in the definition of the $\beta_{i}$ we have had to replace $1/r(x)$ by $1/\Delta(x) = g(x)/r(x)^m$. Note
that the $\alpha_{i}$ have not changed and still converge to $\Frob_p(1/r(x))$.
\end{rem}

\begin{prop} \label{prop:con}
Let $G^0 \in M_{d_x \times d_x}(\ZZ_q[x,1/r])$ and $G^{\infty} \in M_{d_x \times d_x}(\ZZ_q[x,1/x,1/r])$ denote the matrices such that
\begin{align*}
db^0_j &= \sum_{i=0}^{d_x-1} G^0_{i+1,j+1} b^0_{i} dx,
&db^{\infty}_j &= \sum_{i=0}^{d_x-1} G^{\infty}_{i+1,j+1} b^{\infty}_{i} dx,
\end{align*}
for all $0 \leq j \leq d_x-1$. Let $x_0 \neq \infty$ be a geometric point of $\PP^1(\bar{\QQ}_q)$. Then the matrix $G^0 dx$ has at most a simple pole at $x_0$. 
Similarly, the matrix $G^{\infty} dx$ has at most a simple pole at $x=\infty$.
\end{prop}
\begin{proof}
For $G^{\infty} dx$ the proof is given in~\cite[Proposition $2.8$]{tuitman}. For $G^0 dx$ the argument is the same, 
replacing the integral basis $b^{\infty}$ by $b^0$ and the local parameter $t$ by $(x-x_0)$.
\end{proof}

In particular, we have that $rG^0 \in M_{d_x \times d_x}(\ZZ_q[x])$. 

\begin{defn}
Let $x_0 \in \PP^1(\bar{\QQ}_q) \setminus \infty$ be a geometric point. The exponents of $G^0 dx$ at $x_0$ are defined as the
eigenvalues of the residue matrix $G^{x_0}_{-1}=(x-x_0)G^0|_{x=x_0}$. Moreover, the exponents of $G^{\infty} dx$ at $x=\infty$ are defined as 
its exponents at $t=0$, after substituting $x=1/t$.
\end{defn}

\begin{prop} \label{prop:exps} The exponents of $G^0 dx$ at any geometric point $x_0 \in \PP^1(\bar{\QQ}_q) \setminus \infty$ 
and the exponents of $G^{\infty} dx$ at $x = \infty$ are elements of $\QQ \cap \ZZ_p$ and are contained in the interval $[0,1)$. 
\end{prop}

\begin{proof} 
The proof is the same as that of~\cite[Proposition 2.10]{tuitman} replacing the integral basis $[1,y,\ldots,y^{d_x-1}]$ by 
$[b^0_0,\ldots,b^0_{d_x-1}]$.
\end{proof}

\begin{defn} For a geometric point $x_0 \in \mathbf{P}^1(\bar{\QQ}_q)$, we let $\ord_{x_0}(\cdot)$ denote the discrete valuation on
$\bar{\QQ}_q(x)$ corresponding to $x_0$. Moreover, we define 
\[
\ord_{\neq \infty}(\cdot) = \min_{x_0 \in \PP^1(\bar{\QQ}_q) \setminus \infty} \{\ord_{x_0}(\cdot)\}.
\]
We extend these definitions to matrices over $\bar{\QQ}_q(x)$ by taking the minimum over their entries.
\end{defn}

\begin{prop} \label{prop:convbound}
Let $N \in \NN$ be a positive integer.
\begin{enumerate}
\item The element $\Frob_p(1/r)$ of $\mathcal{S}^{\dag}$ is congruent modulo~$p^N$ to 
\[\sum_{i=p}^{pN} \frac{\rho_i(x)}{r^i},
\]
 where  
$\rho_i \in \ZZ_q[x]$ satisfies $\deg(\rho_i)<\deg(r)$ for all $p \leq i \leq pN$.
\item For all $0 \leq i \leq d_x-1$, the element $\Frob_p(y^i)$ of $\mathcal{R}^{\dag}$ is congruent modulo~$p^N$ to 
$\sum_{j=0}^{d-1} \phi_{i,j}(x) y^j$, where
\[
\phi_{i,j} = \sum_{k=0}^{p(N-1)-\ord_{\neq \infty}(W^0)-p\ord_{\neq \infty}((W^0)^{-1})} \frac{\phi_{i,j,k}(x)}{r^k}
\]
for all $0 \leq j \leq d_x-1$ and $\phi_{i,j,k} \in \ZZ_q[x]$ satisfies
\begin{align*}
\deg(\phi_{i,j,0})& \leq -\ord_{\infty}(W^{\infty})-p\ord_{\infty}((W^{\infty})^{-1}), \\
\deg(\phi_{i,j,k})& <\deg(r),
\end{align*}
for all $0 \leq j \leq d_x-1$ and $1 \leq k \leq p(N-1)-\ord_{\neq \infty}(W^0)-p\ord_{\neq \infty}((W^0)^{-1})$.
\item  For all $0 \leq i \leq d_x-1$, the element $\Frob_p(b^0_i/r)$ of $\mathcal{R}^{\dag}$ is congruent modulo~$p^N$ 
to  $\sum_{j=0}^{d_x-1} \psi_{i,j}(x) (b^0_j/r)$, where
\[
\psi_{i,j} = \sum_{k=0}^{pN-1} \frac{\psi_{i,j,k}(x)}{r^k}
\]
for all $0 \leq j \leq d_x-1$ and $\psi_{i,j,k} \in \ZZ_q[x]$ satisfies
\begin{align*}
\deg(\psi_{i,j,0})&\leq -\ord_{\infty}(W)-p\ord_{\infty}(W^{-1})-(p-1)\deg(r), \\
\deg(\psi_{i,j,k})&<\deg(r),
\end{align*}
for all $0 \leq j \leq d_x-1$ and $1 \leq k \leq pN-1$.
\end{enumerate}
\end{prop}

\begin{proof}
The proof is very similar to that of~\cite[Proposition 2.12]{tuitman}.
\end{proof}

\section{Computing (in) the cohomology}

\label{sec:coho}

\begin{defn} The rigid cohomology of $U$ in degree $1$ can be defined as
\[
\Hrig^1(U) = \fCoKer(d: \mathcal{R}^{\dag} \to \Omega^1(\mathbb{U}) \otimes \mathcal{R}^{\dag}). 
\]
\end{defn}

\begin{thm} \label{thm:comparison}
\begin{align*}
\Hrig^1(U) \cong \HdR^1(\mathbb{U})
\end{align*}
\end{thm}

\begin{proof}
This follows as a special case from the comparison theorem between rigid and de Rham
cohomology of Baldassarri and Chiarellotto \cite{baldachiar}, since by Proposition~\ref{prop:goodlift}
$\mathcal{D}_{\mathcal{X}}$ is smooth over $\ZZ_q$.
\end{proof}

We can effectively reduce any $1$-form to one of low pole order using linear algebra as in~\cite{tuitman}. 
The procedure consists of two parts, the finite reductions at the points not lying over $x=\infty$ and 
the infinite reductions at the points lying over $x=\infty$, respectively. We start with the finite reductions.

\begin{prop} \label{prop:finitered}
For all $\ell \in \NN$ and every vector $w \in \QQ_q[x]^{\oplus d_x}$, 
there exist vectors $u,v \in \QQ_q[x]^{\oplus d_x}$
with $\deg(v) < \deg(r)$, such that
\begin{align*}
\frac{\sum_{i=0}^{d_x-1} w_i b^0_i}{r^{\ell}} \frac{dx}{r} &= d\left(\frac{\sum_{i=0}^{d_x-1} v_i b^0_i}{r^{\ell}} \right)+\frac{\sum_{i=0}^{d_x-1}u_i b^0_i}{r^{\ell-1}} \frac{dx}{r}.
\end{align*}
\end{prop}

\begin{proof}
The proof is the same as that of~\cite[Proposition $3.3$]{tuitman} replacing the integral basis $[1,y,\ldots,y^{d_x-1}]$ by 
$[b^0_0,\ldots,b^0_{d_x-1}]$.
\end{proof}
We now move on to the infinite reductions.

\begin{prop} \label{prop:infinitered}
For every vector $w \in \QQ_q[x,1/x]^{\oplus d_x}$ with 
\[
\ord_{\infty}(w) \leq - \deg(r),
\] 
there exist vectors $u,v \in \QQ_q[x,1/x]^{\oplus d_x}$ with $\ord_{\infty}(u) > \ord_{\infty}(w)$ such that
\begin{align*}
\left(\sum_{i=0}^{d_x-1} w_i b^{\infty}_i\right) \frac{dx}{r} =  d\left(\sum_{i=0}^{d_x-1}v_i b^{\infty}_i\right)+\left(\sum_{i=0}^{d_x-1}u_i b^{\infty}_i \right) \frac{dx}{r}.
\end{align*}
\end{prop}

\begin{proof} 
The proof is given in~\cite[Proposition $3.4$]{tuitman}
\end{proof}

\begin{rem} \label{rem:ord0W}
Note that when $\ord_{\infty}(w) \leq \ord_0(W)-\deg(r)+1$, we have that $\ord_0(v) \geq -\ord_0(W)$, so that the 
function $\sum_{i=0}^{d_x-1}v_i b^{\infty}_i$ only has poles at points lying over $x=\infty$.
\end{rem}

Next we give an explicit description of the cohomology space $\Hrig^1(U)$.

\begin{thm} \label{thm:cohobasis}
Define the following $\QQ_q$-vector spaces:
\begin{align*}
E_0 &= \Bigg\{ \left( \sum_{i=0}^{d_x-1} u_i(x) b^0_i \right) \frac{dx}{r} &\colon& u \in \QQ_q[x]^{\oplus d_x} \Bigg \}, \\
E_{\infty}&= \Bigg \{ \left( \sum_{i=0}^{d_x-1} u_i(x,1/x) b_{i}^{\infty} \right) \frac{dx}{r} &\colon& u \in \QQ_q[x,1/x]^{\oplus d_x}, \ord_{\infty}(u) > \ord_0(W)-\deg(r)+1 \Bigg \}, \\
B_0 &= \bigg \{ \sum_{i=0}^{d_x-1} v_i(x) b^0_i &\colon& v \in  \QQ_q[x]^{\oplus d_x} \bigg \}, \\ 
B_{\infty}&=\bigg \{ \sum_{i=0}^{d_x-1} v_i(x,1/x) b^{\infty}_i &\colon& v \in \QQ_q[x,1/x]^{\oplus d_x}, \ord_{\infty}(v) > \ord_0(W) \bigg \}.
\end{align*}
Then $E_0 \cap E_{\infty}$ and $d(B_0 \cap B_{\infty})$ are finite dimensional $\QQ_q$-vector spaces and 
\begin{align*}
\Hrig^1(U) &\cong (E_0 \cap E_{\infty})/d(B_0 \cap B_{\infty}).
\end{align*}
\end{thm}

\begin{proof}
The proof is the same as that of~\cite[Theorem 3.6]{tuitman} replacing the change of basis matrix $W^{\infty}$ by $W$.
\end{proof}

Note that by the proof of Theorem~\ref{thm:cohobasis}, we can effectively reduce any $1$-form to one
in $E_0 \cap E_{\infty}$ with the same cohomology class. However, the reduction procedure will introduce 
$p$-adic denominators and therefore suffer from loss of $p$-adic precision. In the following two propositions 
we bound these denominators.

\begin{prop} \label{prop:finiteprecision}
Let $\omega \in \Omega^1(\mathcal{U})$ be of the form
\[
\omega=\frac{\sum_{i=0}^{d_x-1} w_i b^0_i}{r^{\ell}} \frac{dx}{r},
\]
where $\ell \in \NN$ and $w \in \ZZ_q[x]^{\oplus d_x}$ satisfies $\deg(w)<\deg(r)$.
We define 
\[
e_0 = \max \{ e_P | P \in \mathcal{X} \setminus \mathcal{U}, x(P) \neq \infty \}.
\]
If we represent the class of $\omega$ in $\Hrig^1(U)$ by  
\[
\left(\sum_{i=0}^{d_x-1} u_i b^0_i \right) \frac{dx}{r},
\]
with $u \in \QQ_q[x]^{\oplus d_x}$ as in the proof of Theorem~\ref{thm:cohobasis}, then
\[
p^{\lfloor \log_p(\ell e_0) \rfloor} u \in \ZZ_q[x]^{\oplus d_x}.
\]
\end{prop}

\begin{proof}
The proof is the same as that of~\cite[Proposition $3.7$]{tuitman} replacing the 
integral basis $[1,y,\ldots,y^{d_x-1}]$ by $[b^0_0,\ldots,b^0_{d_x-1}]$.
\end{proof}

\begin{prop} \label{prop:infiniteprecision}
Let $\omega \in \Omega^1(\mathcal{U})$ be of the form
\[
\omega=(\sum_{i=0}^{d_x-1} w_i(x,x^{-1}) b_i^{\infty}) \frac{dx}{r},
\]
where $w \in \ZZ_q[x,x^{-1}]^{\oplus d_x}$ satisfies $\ord_{\infty}(w) \leq \ord_0(W^{\infty})-\deg(r)+1$. 
We write $m = -\ord_{\infty}(w)-\deg(r)+1$ and define
\begin{align*}
e_{\infty}      &= \max \{ e_P | P \in \mathcal{X} \setminus \mathcal{U}, x(P) = \infty \}. 
\end{align*}
If we represent the class of $\omega$ in $\Hrig^1(U)$ by  
\[
\left(\sum_{i=0}^{d_x-1} u_i b^{\infty}_i \right) \frac{dx}{r},
\]
with $u \in \QQ_q[x,x^{-1}]^{\oplus d_x}$ such that $\ord_{\infty}(u) > \ord_0(W^{\infty})-\deg(r)+1$ as in the proof of Theorem~\ref{thm:cohobasis}, then
\[
p^{\lfloor \log_p(m e_{\infty}) \rfloor} u \in \ZZ_q[x,x^{-1}]^{\oplus d_x}.
\]
\end{prop}

\begin{proof}
The proof is given in~\cite[Proposition $3.8$]{tuitman}
\end{proof}

\begin{rem}
Note that Propositions~\ref{prop:finitered}, \ref{prop:infinitered}, \ref{prop:finiteprecision} and \ref{prop:infiniteprecision}
can be used to give an alternative effective proof of Theorem~\ref{thm:comparison}.
\end{rem}

Recall that in Theorem~\ref{thm:cohobasis} the computation of a basis for
$\Hrig^1(U)$ was reduced to a finite dimensional linear algebra problem. However, the dimension of $\Hrig^1(U)$ is generally 
much higher than the dimension of $\Hrig^1(X)$, so that we would like to compute a basis for this last space. For this we will need 
to compute the kernel of a cohomological residue map. 

\begin{defn}
For a $1$-form $\omega \in \Omega^1(\mathcal{U})$ and a point $P \in \mathcal{X} \setminus \mathcal{U}$, 
we let 
\[ 
res_P(\omega) \in \mathcal{O}_{\mathcal{X},P}/(z_P)
\]
denote the coefficient $a_{-1}$ in the Laurent series expansion
\[
\omega = (a_{-k} z_P^k + \dotsc + a_{-1} z_P^{-1}+ \cdots) dz_P.
\] 
Moreover, we denote
\begin{align*}
res_0          &= \bigoplus_{P \in \mathcal{X} \setminus \mathcal{U} \colon x(P) \neq \infty} res_P, 
&res_{\infty}  &= \bigoplus_{P \in \mathcal{X} \setminus \mathcal{U} \colon x(P) = \infty} res_P.
\end{align*}
\end{defn}

\begin{thm} We have an exact sequence 
\[
\begin{CD}
0 @>>> \Hrig^1(X) @>>> \Hrig^1(U) @>(res_0 \oplus res_{\infty}) \otimes \QQ_q>> \underset{P \in \mathcal{X} \setminus \mathcal{U}}{\bigoplus} \mathcal{O}_{\mathcal{X},P}/(z_P) \otimes \QQ_q.
\end{CD}
\]
\end{thm}
\begin{proof}
This is well known (but hard to find in the literature).
\end{proof}

The kernels of $res_0$ and $res_{\infty}$ can be computed without having to compute the
Laurent series expansions at all $P \in \mathcal{X} \setminus \mathcal{U}$ using the
following two propositions. We start with the infinite residues.

\begin{prop} \label{prop:kerresinfty}
Let $\omega \in \Omega^1(\mathbb{U})$ be a $1$-form of the form 
\[
\omega=\left(\sum_{i=0}^{d_x-1}u_i(x,x^{-1}) b_i^{\infty} \right) \frac{dx}{r},
\]
where $u \in \QQ_q[x,x^{-1}]^{\oplus d_x}$ satisfies $\ord_{\infty}(u)>-\deg(r)$, and let a vector
$v \in \QQ_q^{\oplus d_x}$ be defined by $v= \left(x^{1-\deg(r)}u \right)\lvert_{x=\infty}$. 
Moreover, let the residue matrix $G^{\infty}_{-1} \in M_{d_x \times d_x}(\QQ_q)$ be defined 
as in Proposition~\ref{prop:infinitered}, and let $\mathcal{E}^{\infty}_\lambda$ denote the 
(generalised) eigenspace of $G^{\infty}_{-1}$ with eigenvalue $\lambda$, so that $\QQ_q^{\oplus d_x}$ 
decomposes as $\bigoplus \mathcal{E}^{\infty}_{\lambda}$. Then
\[
res_{\infty}(\omega)=0  \; \; \; \Leftrightarrow \; \; \; \mbox{the projection of $v$ onto $\mathcal{E}^{\infty}_0$ vanishes}.
\]
\end{prop}

\begin{proof}
The proof is given in~\cite[Proposition 3.13]{tuitman}.
\end{proof}

We now move on to the finite residues.

\begin{prop} \label{prop:kerres} Let $\omega \in \Omega^1(\mathbb{U})$ be a $1$-form of the form 
\[
\omega=\left(\sum_{i=0}^{d_x-1}u_i(x) b^0_i \right) \frac{dx}{r},
\]
with $u \in \QQ_q[x]^{\oplus d_x}$. For every geometric point $x_0 \in \mathcal{D}_{\PP^1}(\bar{\QQ}_q) \setminus \infty$, let the vector $v_{x_0} \in \bar{\QQ}_q^{\oplus d_x}$ 
be defined by $v_{x_0}=u\lvert_{x=x_0}$. Moreover, let the residue matrix $G_{-1}^{x_0} \in M_{d_x \times d_x}(\bar{\QQ}_q)$ be defined as 
$G_{-1}^{x_0}=(x-x_0)G^0\lvert_{x=x_0}$, and let $\mathcal{E}^{x_0}_\lambda$ denote the (generalised) eigenspace of $G^{x_0}_{-1}$ with eigenvalue $\lambda$, 
so that $\bar{\QQ}_q^{\oplus d_x}$ decomposes as $\bigoplus \mathcal{E}^{x_0}_{\lambda}$. Then 
\begin{align*}
res_0(\omega)=0 \; \; \; \Leftrightarrow \; \; \; &\mbox{the projection of $v_{x_0}$ onto $\mathcal{E}^{x_0}_0$ vanishes} \\ &\mbox{for all } x_0 \in \mathcal{D}_{\PP^1}(\bar{\QQ}_q) \setminus \infty.
\end{align*}
\end{prop}

\begin{proof}
The proof is completely analogous to that of Proposition~\ref{prop:kerresinfty}.
\end{proof}

\section{The complete algorithm and its complexity}

\label{sec:complete}

In this section we describe all the steps in the algorithm and determine bounds for the complexity. Recall that $X$ is
a curve of genus $g$ over a finite field $\FF_q$ with $q=p^n$ and that $d_x$ and $d_y$ denote the degrees of the defining
polynomial $Q$ in the variables $y$ and $x$, respectively.
All computations are carried out to $p$-adic precision $N$ which will be specified later.  We use the
$\SoftOh(-)$ notation that ignores logarithmic factors, i.e. $\SoftOh(f)$ denotes the class of functions that
lie in $\BigOh(f \log^k(f))$ for some $k \in \NN$. For example, two elements of $\ZZ_q$ can be multiplied 
in time $\SoftOh(\log(p)nN)$. We let $\theta$ denote an exponent for matrix multiplication, so that two 
$k \times k$ matrices can be multiplied in $\BigOh(k^{\theta})$ ring operations.
It is known that $\theta \geq 2$ and that one can take $\theta \leq 2.3729$ \cite{williams2012}. We start with some bounds that will 
be useful later on.

\begin{prop} \label{prop:degr} Let $\Delta$, $s$, $r$ be defined as in Section~\ref{sec:lift} and $e_0, e_{\infty}$
as in Section~\ref{sec:coho}. We have:
\begin{subequations}
\begin{alignat}{3}
\deg(\Delta), \deg(r), \deg(s) & \leq 2(d_x-1) d_y   &\; \in \;& \BigOh(d_x d_y), \label{eq:bound1} \\
e_0, e_{\infty}                  & \leq  d_x              &\; \in \;& \BigOh(d_x), \label{eq:bound3} \\
g                              & \leq (d_x-1)(d_y-1) &\; \in \;& \BigOh(d_x d_y). \label{eq:bound4}
\end{alignat}
\end{subequations}
\end{prop}

\begin{proof} 
The proof is given in~\cite[Proposition $4.1$]{tuitman}.
\end{proof} 

Since in Assumption~\ref{assump:goodlift} we assumed that the matrices $W^0,W^{\infty}$ were given to us,
we cannot say much about their pole orders $\ord_0, \ord_{\infty}$ and $\ord_{\neq \infty}$. However, for 
a rigorous complexity analysis we need some bounds:
\begin{assump} \label{assump:ordW}
We will assume that $-\ord_0, -\ord_{\infty}, -\ord_{\neq \infty}$ of the matrices $W^0, W^{\infty}$ 
and their inverses are contained in $\BigOh(d_x d_y)$.  
\end{assump}
Note that this is a very reasonable assumption, since the matrices $W^0,W^{\infty}$ returned by (for example) 
the algorithm from~\cite{hess} satisfy it. Indeed,  $(W^0)^{-1}$ can be chosen such that the entries 
and the determinant are all polynomials of degree $\BigOh(\deg(\Delta))$ and $W^{-1}=(W^{\infty})^{-1} W^0$ can be 
chosen to be diagonal and such that the entries are all monomials of degree $\BigOh(\deg(\Delta))$. Therefore,
by Proposition~\ref{prop:degr} we have that Assumption~\ref{assump:ordW} is satisfied.

\subsection{Step I: Determine a basis for the cohomology} \mbox{ } \\

We want to find $\omega_1,\dotsc,\omega_{\kappa} \in (E_0 \cap E_{\infty}) \cap \Omega^1(\mathcal{U})$ such that:
\begin{enumerate}
\item $[\omega_1,\dotsc,\omega_{\kappa}]$ is a basis for $\Hrig^1(U) \cong (E_0 \cap E_{\infty})/d(B_0 \cap B_{\infty})$,
\item the class of every element of $(E_0 \cap E_{\infty}) \cap \Omega^1(\mathcal{U})$ in $\Hrig^1(U)$ has $p$-adically integral 
      coordinates with respect to $[\omega_1,\dotsc,\omega_{\kappa}]$,
\item $[\omega_1,\dotsc,\omega_{2g}]$ is a basis for the kernel of $res_0 \oplus res_{\infty}$ and hence for the subspace $\Hrig^1(X)$
      of $\Hrig^1(U)$.
\end{enumerate} 
\begin{proof}
The only difference with~\cite[Section 4.1]{tuitman} is that for an element
\[
\left( \sum_{i=0}^{d_x-1} u_i(x) b^0_i \right) \frac{dx}{r} \in E_0 \cap E_{\infty},
\]
we now have that $\deg(u) \leq \deg(r)-2-\ord_0(W)-\ord_{\infty}(W)$, but this is still $\BigOh(d_x d_y)$ by Assumption~\ref{assump:ordW}. 
Therefore, the time complexity of this step remains 
\[
\SoftOh \left(\log(p) d_x^{2\theta} d_y^{\theta} n N \right). \qedhere
\] 
\end{proof}

\subsection{Step II: Compute the map $\Frob_p$} \mbox{ } \\

We use Theorem~\ref{thm:froblift} to compute approximations:
\begin{align*}
\Frob_p(1/r) &= \alpha_i+ \mathcal{O}(p^{2^i}), \\
\Frob_p(y)           &= \beta_i+ \mathcal{O}(p^{2^i}),
\end{align*}
for $i=1, \dotsc, \nu=\lceil \log_2(N) \rceil$ as in~\cite[Section 4.2]{tuitman}. 
We again carry out all computations using $r$-adic expansions (and not $\Delta$-adic ones!) for the elements of 
$\mathcal{R}$ and $\mathcal{S}$. Note that by Proposition~\ref{prop:convbound} and Assumption~\ref{assump:ordW}, a ring operation in 
$\mathcal{R}$ still takes time $\SoftOh(pd_x^2 d_y (N+d_x) nN)$ and a ring operation in $\mathcal{S}$ time $\SoftOh(pd_x d_y (N+d_x) nN)$. 
Recall that the image of an element of $\QQ_q$ under $\sigma$ can be computed in time $\SoftOh(\log^2(p)n+\log(p)nN)$ by~\cite{hubrechts}. 
As in~\cite[Section 4.2]{tuitman} $(\alpha_{\nu},\beta_{\nu})$ can therefore be computed in time $\SoftOh\left(p d_x^3 d_y  \bigl(N+d_x \bigr) n N\right)$. 

Let $\Phi,\Psi \in M_{d_x \times d_x}(\mathcal{S}^{\dag})$ be the matrices of $\Frob_p$ on $\mathcal{R}^{\dag}$ with respect to the bases 
$[1,y,\ldots,y^{d_x-1}]$ and $[b^0_0/r,\ldots,b^0_{d_x-1}/r]$ over $\mathcal{S}^{\dag}$, respectively. Note that this notation is consistent
with that of Proposition~\ref{prop:convbound}. Then $\Phi$ can be computed from $\beta_{\nu}$ using $\BigOh(d_x)$ ring operations in $\mathcal{R}$. 
Moreover, it follows from the formula
\[
\Psi = (W^0/r) \Phi ((W^{0})^{-1} r)^{F_p}
\]
and Assumption~\ref{assump:ordW}, that $\Psi$ can be computed from $\Phi$ and $\alpha_{\nu}$ 
using $\BigOh(d_x^{\theta} + d_x d_y)$ ring operations in $\mathcal{S}$ and 
$\BigOh\left((d_x d_y) \deg(r) d_x^2 \right) \subset \BigOh(d_x^4 d_y^2)$
applications of $\sigma$. Therefore, the matrix $\Psi$ can be computed from $(\alpha_{\nu},\beta_{\nu})$ in time 
$\SoftOh\left(pd_x^{\theta+1} d_y^2 (N+d_x) nN\right)$. Note that having to compute the matrix $\Psi$ is the main difference
compared to~\cite[Section 4.2]{tuitman}.

Finally, for each $\omega_i=\left(\sum_{k=0}^{d-1} u_k(x) b^0_k \right)dx/r$ with $1 \leq i \leq 2g$, we compute
\begin{align} \label{eq:Fp1}
\Frob_p (\omega_i)
&=\sum_{j=0}^{d_x-1} \left( \sum_{k=0}^{d_x-1} p x^{p-1} u_k^{\sigma}(x^p) \psi_{j,k} \right) b^0_j \frac{dx}{r} + \BigOh\left(p^N\right).
\end{align}
For a single $\omega_i$ this takes $\BigOh(d_x^2)$ ring operations in $\mathcal{S}$ and 
\[
\BigOh\left(d_x \left(\deg(r)-2-\ord_0(W)-\ord_{\infty}(W) \right) \right) \subset \BigOh(d_x^3 d_y)
\]
applications of $\sigma$. Hence the complete set of $\Frob_p(\omega_i)$ can still be computed in time
\[
\SoftOh\left(g p d_x^3 d_y  \left(N+d_x \right) n N\right) \subset \SoftOh\left(p d_x^4 d_y^2  \left(N+d_x \right) n N\right),
\]
which also remains the time complexity of this step.

\subsection{Step III: Reduce back to the basis} \mbox{ } \\

We want to find the matrix $\mathcal{F} \in M_{2g \times 2g}(\QQ_q)$ such that
\[
\Frob_p(\omega_i) = \sum_{j=1}^{2g} \mathcal{F}_{j,i} \omega_j
\]
in $\Hrig^1(U)$. In the previous step, we have obtained an approximation
\begin{equation} 
\Frob_p(\omega_i) = \sum_{j \in J} \left( \sum_{k=0}^{d_x-1} \frac{w_{i,j,k}(x)}{r^j} b^0_k \right) \frac{dx}{r} +\BigOh\left(p^N\right),
\end{equation}
where $J \subset \ZZ$ is finite and $w_{i,j,k}(x) \in \ZZ_q[x]$ satisfies
$\deg(w_{i,j,k}(x))< \deg(r)$ for all $i,j,k$. We now use Proposition~\ref{prop:finitered} and Proposition~\ref{prop:infinitered}
(repeatedly) to reduce this $1$-form to an element of $E_0 \cap E_{\infty}$ as in Theorem~\ref{thm:cohobasis}.

To carry out the reduction procedure, it is sufficient to solve 
a linear system with parameter ($\ell$ or $m$, respectively) only once in 
Propositions~\ref{prop:finitered} and~\ref{prop:infinitered}.
After that, every reduction step corresponds to a multiplication of a vector by a $d_x \times d_x$ matrix 
(over $\QQ_q[x]/(r)$ or $\QQ_q$, respectively). 

The time complexity is the same as in~\cite[Section 4.3]{tuitman}, with only one small correction: by Assumption~\ref{assump:ordW},
the number of reductions steps at the points lying over $x=\infty$ is $\BigOh(pd_x d_y)$, so that all $\Frob_p(\omega_i)$ can 
be reduced in time $\SoftOh(p d_x^4 d_y^2 nN)$. Therefore, the time complexity of this step remains
\[
\SoftOh(pd_x^4 d_y^2 n N^2+d_x^5 d_y^3 n N).
\]

\subsection{Step IV: Determine $Z(X,T)$} \mbox{ } \\

It follows from the Lefschetz formula for rigid cohomology that
\begin{align*}
Z(X,T)                   &= \frac{\chi(T)}{(1-T)(1-qT)}, 
\intertext{where we have} 
\chi(T) &= \det\bigl(1-\Frob_p^n T|\Hrig^1(X) \bigr).
\end{align*}
This polynomial can be computed exactly the same way as in~\cite[Section 4.4]{tuitman}, so that the time complexity of 
this step is still 
\[
\SoftOh(\log^2(p) g^{\theta} n N) \subset \SoftOh(\log^2(p)(d_x d_y)^{\theta} n N).
\]

\subsection{The $p$-adic precision} \mbox{ } \\ \label{sec:precision}

So far we have only obtained an approximation to $\chi(T)$, since 
we have computed to $p$-adic precision $N$. Moreover, because of loss of precision
in the computation, in general $\chi(T)$ will not even be correct to precision~$N$. 
So what precision~$N$ is sufficient to determine $\chi(T)$ exactly? Although the 
bounds used in~\cite{tuitman} were good enough to obtain the right complexity estimate,
they were sometimes not sharp enough in practice. In this section we will carry 
out a much more detailed analysis and will obtain bounds that are usually sharp.

\begin{prop} \label{prop:newgir} In order to recover $\chi(T) \in \ZZ[T]$ exactly, it is sufficient to know it to $p$-adic precision
\begin{align*}
\max_{1 \leq i \leq g} \left\{ \left\lfloor \log_p\left(\frac{4g}{i}\right) + \left(\frac{ni}{2}\right) \right\rfloor + 1 \right\} 
\; \; \; \in \BigOh(d_x d_y n). 
\end{align*}
\end{prop}

\begin{proof}
The expression for the precision is a straightforward consequence of a result of Kedlaya, which can be found in \cite[Lemma 1.2.3]{cycliccubic}. 
That this precision is $\BigOh(d_x d_y n)$ follows from the bound on $g$ from Proposition~\ref{prop:degr}.
\end{proof}

\begin{defn} Let $H^1_{cris}(\mathcal{X},\mathcal{D}_{\mathcal{X}})$ denote the log-crystalline cohomology of $\mathcal{X}$ along the divisor $\mathcal{D}_{\mathcal{X}}$.
We define the following $\ZZ_q$-lattices in $\Hrig^1(U)$:
\begin{align*}
\Lambda_{E_0 \cap E_{\infty}} &=  \fIm \left( (E_0 \cap E_{\infty}) \cap \Omega^1(\mathcal{U})  \rightarrow \Hrig^1(U) \right), \\
\Lambda_{cris}                &=  \fIm \left( H^1_{cris}(\mathcal{X},\mathcal{D}_{\mathcal{X}}) \rightarrow \Hrig^1(U) \right).
\end{align*}
\end{defn}

\begin{defn} Let us denote
\begin{align*}
\delta_1 &= \big\lfloor \log_p \bigl( -(\ord_0(W)+1) e_{\infty} \bigl) \big\rfloor, \\
\delta_2 &=\big\lfloor \log_p\bigl( (\lfloor (2g-2)/d_x \rfloor+1) e_{\infty} \bigr) \big\rfloor, \\
\delta   &=\delta_1+\delta_2.
\end{align*}
\end{defn}

\begin{prop} \label{prop:lattice} We have the following inclusions of lattices:
\[p^{\delta_1} \Lambda_{E_0 \cap E_{\infty}} \subset \Lambda_{cris} \subset p^{-\delta_2} \Lambda_{E_0 \cap E_{\infty}}.\]
\end{prop}
\begin{proof} Our proof generalises that of \cite[Proposition $5.3.1$]{edixhoven}. We define the effective divisor
\begin{align*}
\mathcal{D}_{\infty} = \sum_{P \in \mathcal{X} \setminus \mathcal{U} : x(P) = \infty} e_P P
\end{align*}
on the curve $\mathcal{X}$. For any integer $m \geq 0$, we let $\mathcal{C}^{\bullet}(m)$ denote the complex
\[
\begin{CD}
\mathcal{O}(m \mathcal{D}_{\infty}) @>>> \Omega^1(\log(\mathcal{D}_{\mathcal{X}})) \otimes \mathcal{O}(m \mathcal{D}_{\infty}),
\end{CD}
\]
i.e. the De Rham complex on $\mathcal{X}$ with logarithmic poles along $\mathcal{D}_{\mathcal{X}}$ twisted by the line bundle 
$\mathcal{O}(m \mathcal{D}_{\infty})$. Note that $\mathcal{C}^{\bullet}(l)$ is a subcomplex of $\mathcal{C}^{\bullet}(m)$ whenever 
$l \leq m$. From the comparison theorem between log-De Rham and log-crystalline cohomology, we know that 
$\mathbb{H}^1(\mathcal{C}^{\bullet}(0))=H^1_{cris}(\mathcal{X},\mathcal{D}_{\mathcal{X}})$. 

Recall that $z_P$ denotes an \'etale local coordinate at $P \in \mathcal{X} \setminus \mathcal{U}$. 
For any integer $m \geq 0$, we have the following diagram:
\[
\begin{CD}
H^0(\Omega^1(\log(\mathcal{D}_{\mathcal{X}}))) @>>> \mathbb{H}^1(\mathcal{C}(0)) @>>> H^1(\mathcal{O}) \\
@VVV @VVV @VVV \\
H^0(\Omega^1(\log(\mathcal{D}_{\mathcal{X}})) \otimes \mathcal{O}(m \mathcal{D}_{\infty})) @>>> \mathbb{H}^1(\mathcal{C}(m)) @>>> H^1(\mathcal{O}(m \mathcal{D}_{\infty})) \\
@VVV @VVV @. \\
\bigoplus\limits_{P \in \mathcal{X} \setminus \mathcal{U}:x(P)=\infty} \frac{z_P^{-me_P} \ZZ_q[[z_P]] \frac{dz_P}{z_P}}{\ZZ_q[[z_P]] \frac{dz_P}{z_P}} @>>> 
\mathop{\bigoplus \bigoplus}\limits_{\substack{P \in \mathcal{X} \setminus \mathcal{U}:x(P)=\infty \\ -me_P \leq i < 0}} (\ZZ_q/i\ZZ_q) z_P^i \frac{dz_P}{z_P} \\
@VVV @VVV \\
0 @. 0
\end{CD}
\]
where the first two rows and columns are exact and all (hyper)cohomology is taken with respect to global sections on $\mathcal{X}$. Hence the cokernel of the map 
$$\mathbb{H}^1(\mathcal{C}(0)) \rightarrow \mathbb{H}^1(\mathcal{C}(m))$$ 
is annihilated by $p^{\lfloor \log_p(m e_{\infty}) \rfloor}$. For $m_1=-(\ord_0(W)+1)$, we
have that 
$$H^0(\Omega^1(\log(\mathcal{D}_{\mathcal{X}})) \otimes \mathcal{O}(m_1 \mathcal{D}_{\infty}))=(E_0 \cap E_{\infty}) \cap \Omega^1(\mathcal{U}).$$ 
Therefore, it follows that $p^{\delta_1} \Lambda_{E_0 \cap E_{\infty}} \subset  \Lambda_{cris}$.

We now prove the other inclusion. For $m_2=\lfloor (2g-2)/d_x \rfloor+1$, it follows from Serre duality that
$H^1(\mathcal{O}(m_2 \mathcal{D}_{\infty}))=H^0(\mathcal{O}(\omega_{\mathcal{X}}-m_2 \mathcal{D}_{\infty}))=0$, since
we have that $\deg(m_2 \mathcal{D}_{\infty})) > 2g-2 = \deg(\omega_{\mathcal{X}})$. So the map
$$H^0(\Omega^1(\log(\mathcal{D}_{\mathcal{X}})) \otimes \mathcal{O}(m_2 \mathcal{D}_{\infty})) \rightarrow \Lambda_{cris}$$
is surjective. However, by Proposition~\ref{prop:infinitered}, the class in $\Hrig^1(U)$ of an element of 
$H^0(\Omega^1(\log(\mathcal{D}_{\mathcal{X}})) \otimes \mathcal{O}(m_2 \mathcal{D}_{\infty}))$
can be represented by an element of $p^{-\delta_2} \Lambda_{E_0 \cap E_{\infty}}$. This finishes the proof. 
\end{proof}

\begin{cor}
We have that $\ord_p(\mathcal{F}) \geq -\delta$. 
\end{cor}
\begin{proof}
Note that $\Lambda_{cris}$ is mapped into itself by $\Frob_p$ and that the
basis $[\omega_1,\ldots,\omega_{\kappa}]$ for $\Hrig^1(U)$ 
is by construction a basis for $\Lambda_{E_0 \cap E_{\infty}}$. Therefore, the result 
follows from Proposition~\ref{prop:lattice}.
\end{proof}

\begin{prop} \label{prop:precF}
In order to recover $\chi(T) \in \ZZ[T]$ exactly, it is sufficient to know the matrix 
$\mathcal{F}$ to $p$-adic precision 
\begin{align*}
\max_{1 \leq i \leq g} \left\{ \left\lfloor \log_p\left(\frac{4g}{i}\right) + \left(\frac{ni}{2}\right) \right\rfloor + 1 \right\}+\delta  \; \; \; \in \BigOh(d_x d_y n).
\end{align*}
\end{prop}
\begin{proof}
We have to compute 
$$\mathcal{F}^{(n)}=\mathcal{F}^{\sigma^{(n-1)}} \mathcal{F}^{\sigma^{(n-2)}} \cdots \mathcal{F}$$ 
and its reverse characteristic polynomial $\chi$. The
basis $[\omega_1,\ldots,\omega_{\kappa}]$ for $\Hrig^1(U)$ that we constructed is a basis for
$\Lambda_{E_0 \cap E_{\infty}}$. Note that with respect to a basis for $\Lambda_{cris}$ there
would be no loss of precision in the computation. Therefore, the result follows from Proposition~\ref{prop:lattice}
by changing basis from $[\omega_1,\ldots,\omega_{\kappa}]$ to a basis for $\Lambda_{cris}$, computing $\chi(T)$ with
respect to this basis, and changing basis back to $[\omega_1,\ldots,\omega_{\kappa}]$.
\end{proof}

\begin{defn} We define $f_1: \NN \rightarrow \ZZ_{\geq 0}$ and $f_2 \in \ZZ_{\geq 0}$ by
\begin{align*}
f_1(m)      &= \big\lfloor \log_p\bigl(p(m-1)e_0\bigr) \big\rfloor + \big \lfloor \log_p \bigl( -(\ord_{\infty}(W^{-1})+1)e_{\infty} \bigr) \big \rfloor,  \\
f_2         &= \big\lfloor \log_p\bigl(-p(\ord_0(W)+1) e_{\infty}) \bigr) \big\rfloor.
\end{align*}
\end{defn}

\begin{prop} \label{prop:prec} In order to recover $\chi(T) \in \ZZ[T]$ exactly, it is sufficient to choose the $p$-adic precision $N$ such that for all $m \geq N$
\[ m-\max \{ f_1(m), f_2 \} \geq \max_{1 \leq i \leq g} \left\{ \left\lfloor \log_p\left(\frac{4g}{i}\right) + \left(\frac{ni}{2}\right) \right\rfloor + 1 \right\}+\delta. \]
Therefore, we may take $N \in \SoftOh(d_x d_y n)$.
\end{prop}

\begin{proof}
We write
\begin{align} \label{eq:padicprecision}
\Frob_p(\omega_i) = \sum_{j \in \ZZ} \left( \sum_{k=0}^{d_x-1} \frac{w_{i,j,k}(x)}{r^j} b^0_k \right) \frac{dx}{r} 
\end{align}
where $w_{i,j,k}(x) \in \ZZ_q[x]$ satisfies $\deg(w_{i,j,k}(x)) < \deg(r)$ for all $i,j,k$. 

First, consider the terms with $j>0$. If $\ord_p(w_{i,j,k})=m$, then we know from 
Proposition~\ref{prop:convbound}, and the factor $p$ appearing in~$\eqref{eq:Fp1}$,
that $j \leq pm$. Therefore, it follows from Proposition~\ref{prop:finiteprecision} that the loss 
of precision during the finite reductions of terms in~\eqref{eq:padicprecision}
with $j>0$ and $p$-adic valuation~$m$ is at most $\big\lfloor \log_p\bigl(p m e_0\bigr) \big\rfloor$. However, the finite reductions
can introduce a (small) pole at the points lying over $\infty$, which still 
has to be reduced as well. The matrix of the change of basis from $[b_0^{0},\ldots,b_{d_x-1}^{0}]$ to $[b_0^{\infty},\ldots,b_{d_x-1}^{\infty}]$ 
is $W^{-1}$ and $\ord_{\infty}(v_i/r^{\ell}) \geq 1$ for all $0 \leq i \leq d_x-1$ and $\ell>0$ in Proposition~\ref{prop:finitered}. 
Therefore, it follows from Proposition~\ref{prop:infiniteprecision} that the loss of precision during these final infinite reductions 
is at most $\big \lfloor \log_p \bigl( -(\ord_{\infty}(W^{-1})+1)e_{\infty} \bigr) \big \rfloor$. We conclude that the 
total loss of precision during the reductions of the terms in~\eqref{eq:padicprecision} with $j>0$ and $p$-adic valuation~$m$ is at most $f_1(m)$.

Second, consider the terms with $j \leq 0$. By the definition of $E_{\infty}$, the coefficients of $\omega_i$ 
with respect to the basis $[b^{\infty}_0,\ldots,b^{\infty}_{d_x-1}] dx/r$ have order
at $\infty$ bounded below by $\ord_0(W)-\deg(r)+2$ for all $0 \leq i \leq d_x-1$. Therefore, 
with respect to the basis $[b^{\infty}_0,\ldots,b^{\infty}_{d_x-1}] dx/x$ the coefficients of $\omega_i$ have order at $\infty$ 
bounded below by $\ord_0(W)+1$. By (the proof of) Proposition~\ref{prop:convbound}, the Frobenius structure on 
$\RR^0 x_{*}(\mathcal{O}_{\mathbb{U}})$ does not have a pole at $\infty$ with respect to the basis 
$[b^{\infty}_0,\ldots,b^{\infty}_{d_x-1}]$. Moreover, note that $\Frob_p$ sends the $1$-form $dx/x$ to $p dx/x$.
Hence the coefficients of $\Frob_p(\omega_i)$ with respect to the basis 
$[b^{\infty}_0,\ldots,b^{\infty}_{d_x-1}]dx/x$ have order at $\infty$ bounded below by $p(\ord_0(W)+1)$.
So the coefficients of $\Frob_p(\omega_i)$ with respect to the basis 
$[b^{\infty}_0,\ldots,b^{\infty}_{d_x-1}]dx/r$ have order at $\infty$ bounded below by $p(\ord_0(W)+1)-\deg(r)+1$.
Therefore, it follows from Proposition~\ref{prop:infiniteprecision} that the loss of precision during the reductions
of the terms in \eqref{eq:padicprecision} with $j<0$ is at most $f_2$.

We conclude that terms in~\eqref{eq:padicprecision} that have $p$-adic valuation $m$ before reduction will have $p$-adic valuation at least
$m-\max \{ f_1(m), f_2 \}$ after reduction. The bound for the $p$-adic precision~$N$ now follows from Proposition~\ref{prop:precF} and is
clearly contained in $\SoftOh(d_x d_y n)$.
\end{proof}

\begin{thm} \label{thm:time}
The time complexity of the algorithm presented in this section is $\SoftOh(p d_x^6 d_y^4 n^3)$. 
\end{thm}

\begin{proof}
We take the sum of the complexities of the different steps using Proposition~\ref{prop:prec}, 
leaving out terms and factors that are absorbed by the $\SoftOh$. 
\end{proof}

For the analysis of the space complexity, we will not go into the same detail as for the time
complexity. However, one can prove the following theorem. 

\begin{thm} \label{thm:space}
The space complexity of the algorithm presented in this section is 
$\SoftOh(p d_x^4 d_y^3 n^3)$. 
\end{thm}

\begin{proof}
The space complexity of the algorithm turns out to be that of storing
a single $\Frob_p(\omega_i)$, or equivalently an element of $\mathcal{R}$, 
which is $\SoftOh(p d_x^2 d_y  \bigl(N+d_x \bigr) n N )$.
The result now follows using Proposition~\ref{prop:prec}.
\end{proof}

\begin{rem} \label{rem:exclude} 
We should mention that we have excluded the computation of the matrices of the maps $res_0$ and 
$res_{\infty}$, or rather the computation of the eigenspaces of the residue matrices $G^{x_0}_{-1}$
and $G^{\infty}_{-1}$ from our complexity estimates. Analysing the available algorithms would take 
us too far, as they involve factorising polynomials etc. In practice, the time spent on computing
these eigenspaces is always neglible.
\end{rem}

\section{Implementation}

We have updated our Magma~\cite{magma} implementation from~\cite{tuitman}. The code can be found in the packages
\verb{pcc_p{ and \verb{pcc_q{ at our webpage\footnotemark \footnotetext{\url{http://perswww.kuleuven.be/jan_tuitman}}. 
We now provide an example that the algorithm from~\cite{tuitman} was not able to handle, mainly to show how to use
the code. Many more interesting examples as well as timings can be found in the example files that come with the packages and 
in~\cite{castrycktuitman}. We used Magma v2.20-3 and \verb{pcc_p-2.14{ for the computation below.

\subsubsection*{Example} The modular curve $X_1(23)$. \medskip 

Sutherland~\cite{sutherland} gives an equation $\mathcal{Q}$ for a singular plane model of the modular curve $X_1(23)$. This equation can be loaded into our
code in the following way: 
\medskip
\begin{Verbatim}[fontsize=\tiny]
load "pcc_p.m";
Q:=y^7+(x^5-x^4+x^3+4*x^2+3)*y^6+(x^7+3*x^5+x^4+5*x^3+7*x^2-4*x+3)*y^5+(2*x^7+3*x^5-x^4-2*x^3-x^2-8*x+1)*y^4+ 
(x^7-4*x^6-5*x^5-6*x^4-6*x^3-2*x^2-3*x)*y^3-(3*x^6-5*x^4-3*x^3-3*x^2-2*x)*y^2+(3*x^5+4*x^4+x)*y-x^2*(x+1)^2;
\end{Verbatim}
\medskip
Note that $d_x=7$, which is known to be optimal~\cite{derickxvanhoeij}. It turns out that $\mathcal{Q}$ satisfies Assumption~\ref{assump:goodlift}
for all prime numbers 
\[p \notin \{ 2,3,23,41,73,83,2039 \}.\] 
To compute the numerator of the zeta function of $X_1(23)$ modulo $11$, we enter the following commands: 
\medskip
\begin{Verbatim}[fontsize=\tiny]
p:=11;
chi:=num_zeta(Q,p:verbose:=true);
\end{Verbatim}
\medskip
The syntax has changed a bit compared to~\cite{tuitman}, since the $p$-adic precision~$N$ has become
an optional parameter. By default the code now handles the $p$-adic precision itself. We find that the numerator $\chi$
of the zeta function is equal to
\medskip
\begin{Verbatim}[fontsize=\tiny]                                                                                                                                                                                             
3138428376721*x^24 - 285311670611*x^23 - 285311670611*x^22 - 51874849202*x^21 - 14147686146*x^20 - 857435524*x^19 + 
8009227281*x^18 - 226759808*x^17 - 248018540*x^16 - 23205985*x^15 - 22356807*x^14 - 4861824*x^13 + 6990592*x^12 - 
441984*x^11 - 184767*x^10 - 17435*x^9 - 16940*x^8 - 1408*x^7 + 4521*x^6 - 44*x^5 - 66*x^4 - 22*x^3 - 11*x^2 - x + 1
\end{Verbatim}

\begin{rem}
There is a more efficient way to compute the zeta function of modular curves modulo a prime number~$p$, 
using modular symbols~\cite[\S 4.2]{bruin}. Again, this example mainly serves to show how to use the code.
\end{rem}

\bibliographystyle{amsplain}
\bibliography{pcc}

\end{document}